\documentclass{birkjour}
\usepackage{color}
\usepackage{verbatim}
 \newtheorem{thm}{Theorem}[section]
 \newtheorem{corollary}[thm]{Corollary}
 \newtheorem{lemma}[thm]{Lemma}
 \newtheorem{Proposition}[thm]{Proposition}
 \theoremstyle{definition}
 
 \theoremstyle{remark}
 \newtheorem{remark}[thm]{Remark}
 
 \numberwithin{equation}{section}

\begin{document}

%
%

\title[Umbilical points of a $W$-congruence]
{\LARGE Stable umbilical points of a  $W$-congruence}

\author[M.Craizer]{Marcos Craizer}

\address{%
Departamento de Matemática- PUC-Rio\br
Rio de Janeiro, RJ, Brasil}
\email{craizer@puc-rio.br}

\author[R.A.Garcia]{Ronaldo Garcia}

\address{%
Instituto de Matemática e Estatística- UFG\br
Goiânia, GO, Brasil}
\email{ragarcia@ufg.br}

\thanks{The authors want to thank CNPq and CAPES (Finance Code 001)  for financial support during the preparation of this manuscript. \newline E-mail of the corresponding author: craizer@puc-rio.br}



\date{December 12, 2025}

\begin{abstract}
The class of $W$-congruences is a central object of Projective Differential Geometry. Nevertheless, their singularities has not been extensively studied. In this paper we prove a characterization of $W$-congruences that allow us to study their umbilical points. We describe examples of isolated stable umbilical points of type $A_m$,  $m\in\{1,2,3,4\}$, and discuss some partial results concerning the generalization of these results to any $m\in\mathbb{N}\cup\{\infty\}$. 
\end{abstract}

\maketitle

\section{Introduction}

A {\it line congruence} is a smooth bi-dimensional collection of lines in $3$-space. It is a central object of {\it Projective Differential Geometry}. 
Line congruences determine a relation between their focal sets, the so-called {\it transformations of surfaces}. When 
this relation preserves the conformal class of the projective metrics, the line congruence is said to be a {\it Weingarten congruence}, shortly, a {\it W-congruence}. 
Classes of projective surfaces are solutions of certain PDE's, many times associated with a physical phenomena. Their transformations describe methods for obtaining new solutions from old ones.  According to \cite{Eisenhart}, every important transformation is a $W$-congruence. Some classical examples of $W$-congruences are
{\it Bäcklund} transformations between pseudo-spheres, {\it Tzitzéica} transformations between affine spheres and {\it Demoulin} transformations of projectively minimal surfaces. The normal lines to surfaces in Weingarten class, which includes constant mean curvature and constant Gaussian curvature, are also examples of $W$-congruences
(\cite{Sasaki}).

In this paper we study umbilical points of a $W$-congruence. When we began the study of such points, we had two problems in mind. The first one is related to Loewner's conjecture (\cite{Titus}). It is proved  in  \cite{Gutierrez-Sotomayor} that Loewner's conjecture holds in the particular case of a normal congruence of a constant mean curvature surface, and we wanted to see if it holds for a general $W$-congruence. The second one is the existence of stable umbilical points in the space of $W$-congruences. This second problem has not an obvious answer, since they do not exist in the space of line congruences (\cite{Craizer-Garcia-1},\cite{Izu}). We have made no progress in the first problem, but in this paper we give a positive answer to the second one.

In order to study the umbilical points of a line congruence, we shall describe each line by its intersection $f(x,y)$ with a transversal plane and its direction vector $\xi(x,y)$. We may assume that the transversal plane is the plane $z=0$ and that the third coordinate of the direction vector is $1$. Under these hypothesis, the line congruence is described by
\begin{equation}\label{eq:Congruence}
f(x,y)=(f^1(x,y),f^2(x,y),0), \ \ \ \xi=(\xi^1(x,y),\xi^2(x,y),1),
\end{equation}
for $(x,y)$ in the domain $\mathcal{D}\subset\mathbb{R}^2$, where $f$ is an immersion and $\xi$ a directional vector of the line. 
For any line congruence defined by Equation \eqref{eq:Congruence}, one can describe the Weingarten condition by an equation of the form $W=0$. By choosing $f(x,y)=(x,y,0)$, we shall prove that
\begin{equation}\label{eq:DefineW}
W= (\xi^2_y-\xi^1_x)(\xi^2_{xx}\xi^1_{yy}-\xi^1_{xx}\xi^2_{yy})-2\xi^1_y(\xi^1_{xx}\xi^2_{xy}-\xi^2_{xx}\xi^1_{xy})+2\xi^2_x(\xi^1_{xy}\xi^2_{yy}-\xi^1_{yy}\xi^2_{xy}).
\end{equation}
Note that, at points where one of the focal sets is singular, the Weingarten condition defined above does not make sense. But one can use Equation \eqref{eq:DefineW} as a definition of the Weingarten condition at singular points that are limits of regular points. 

It is easier to study analytical $W$-conguences than smooth ones, and we shall do this in this paper. 
For an analytic $W$-congruence, the condition $W=0$ may be replaced by the conditions
equations 
\begin{equation}\label{eq:WDerivatives}
W_{n-k,k}=0, 
\end{equation}
$k=0,1,...n$, $n\in\mathbb{N}$, where $W_{n-k,k}$ denotes the derivative of order $n$, $n-k$ times in $x$ and $k$ times in $y$, at the origin. To understand the jet space at an umbilical point of a $W$-congruence, we need to understand the relations between the coefficients of the Taylor expansion of $\xi$ determined by Equations \eqref{eq:WDerivatives}. We were able to understand all of them, except possibly  one, namely $W_{m0}=0$, where $m$ is a parameter defined by the $2$-jet of $\xi$. For $m$ small, say
$m\in\{1, 2,3,4\}$, we could verify that this equation holds automatically, and our conjecture is that this is true for any $m\in\mathbb{N}$. For $m\not\in\mathbb{N}$, the jet space at a non-degenerate umbilical point is completely understood. 

The main result of the paper is that, for $m\in\{1, 2,3,4\}$, the umbilical point is of type $A_m$, for an open and dense subset of $W$-congruences in the $C^{\omega}$-topology, thus giving a positive answer to the second problem described above. 
We conjecture that it is possible to generalize this result for any $m\in\mathbb{N}$, and that when 
$m\not\in\mathbb{N}$, the non-degenerate umbilical point is of type $A_{\infty}$. In order to support these conjectures, 
we describe a class of examples that behaves like that.

The paper is organized as follows: In Section 2, we prove that Weingarten congruences are characterized by the condition $W=0$, where $W$ is given by Equation \eqref{eq:DefineW}.  In Section 3, we discuss the $C^{\omega}$-topology in the space of jets
at umbilical points. In Section 4 we proof our main results, namely, that for $m\in\{1, 2,3,4\}$, the umbilical point is of type $A_m$ for an open and dense set of $W$-congruences in the $C^{\omega}$-topology.  Finally, in Section 5 we discuss some conjectures and examples. 

\section{Characterization of $W$-congruences}

\subsection{Elements of a line congruence}
			
The main elements of a line congruence are the following: A principal line is a curve in the domain whose corresponding ruled surface is developable, and principal directions are tangent to principal lines.
The discriminant set is formed by the points of the domain where both focal sets coincide. An umbilical point is a point of the discriminant set where all directions are principal. 
There are two ridge and two subparabolic sets. The $i$-ridge set, $i=1,2$, is formed by the points of the domain where the $i$-focal set is singular, while the $i$-subparabolic set is formed by the points of the domain where the $i$-focal set is parabolic (\cite{FIRT}).

The ridge and discriminant points of the line congruence are called {\it singular}, while all others are called {\it regular}. In \cite{Craizer-Garcia-1} it is proved that, generically in the space of line congruences, every singular point is a limit of a sequence of regular points. Along this paper, we shall consider only line congruences satisfying this condition. Observe that the Weingarten condition as defined above does not make sense at singular points, but we shall consider below a condition at regular points that can be extended to singular points.
	
Consider a line congruence defined by Equations \eqref{eq:Congruence}. We write
$$
\xi_x=-\left(af_x+cf_y\right), \ \ \ \xi_y=-\left(bf_x+df_y\right).
$$
for some real functions $a,b,c,d$ defined on the domain $\mathcal{D}$.
The linear operator
\[
S=\left[
\begin{array}{cc}
a  &  b \\
c  & d
\end{array}
\right]
\]
is called the shape operator. 
When $f=(x,y,0)$, we have that
$$
a=-\xi^1_x, \ \ b=-\xi^1_y, \ \ c=-\xi^2_x, \ \ d=-\xi^2_y. 
$$
The discriminant $\delta$ is defined by 
$$
\delta(x,y)=(a-d)^2+4bc,
$$
and the discriminant set is defined by $\delta=0$.  A discriminant point is called umbilical when $a=d$ and $b=c=0$.  
We denote by $\lambda_1$ and $\lambda_2$ the eigenvalues of the shape operator. Since the results of this paper are local, we may assume that  $\lambda_i\neq 0$, $i=1,2$.

\subsection{Projective metrics on focal sets}

Consider a non-discriminant point. In a neighborhood of such a point, we may use principal coordinates $(u,v)$, and we shall do this along this subsection. More precisely, we shall assume that $v=constant$ are the $1$-principal lines and $u=constant$ are the $2$-principal lines. The focal sets are then defined by
$$
E_1(u,v)=f(u,v)+\frac{1}{\lambda_1(u,v)}\xi(u,v), \ \ E_2(u,v)=f(u,v)+\frac{1}{\lambda_2(u,v)}\xi(u,v),
$$
where $\lambda_1$ and $\lambda_2$ are the eigenvalues of the shape operator.

\begin{lemma}\label{Lemma:SingularFocalSets}
Consider a non-discriminant point of the line congruence. Then
$$
(E_1)_u=-\frac{(\lambda_1)_u}{\lambda_1^2}\xi, \ \ \ (E_1)_v=\left( 1-\frac{\lambda_2}{\lambda_1} \right)f_v-\frac{(\lambda_1)_v}{\lambda_1^2}\xi,
$$
and
$$
(E_2)_u=\left( 1-\frac{\lambda_1}{\lambda_2} \right)f_u-\frac{(\lambda_2)_u}{\lambda_2^2}\xi, \ \ \ (E_2)_v=-\frac{(\lambda_2)_v}{\lambda_2^2}\xi.
$$
We conclude that $E_1$ is singular if and only if $(\lambda_1)_u=0$ and  $E_2$ is singular if and only if $(\lambda_2)_v=0$.
\end{lemma}

\begin{proof}
Immediate from the equations $\xi_u=-\lambda_1 f_u$ and $\xi_v=-\lambda_2 f_v$.
\end{proof}

\begin{lemma}\label{Lemma:FocalMetric}
Consider a non-discriminant point of the line congruence. Then
$$
\left[ (E_1)_u, (E_1)_v, (E_1)_{uu} \right]=\frac{(\lambda_1)_u^2(\lambda_1-\lambda_2)}{(\lambda_1)^4}\left[f_u,f_v\right],
$$
$$
\left[ (E_1)_u, (E_1)_v, (E_1)_{vv} \right]=-\frac{(\lambda_1)_u(\lambda_1-\lambda_2)^2}{(\lambda_1)^4}\left[f_v,f_{vv}\right],
$$
and $\left[ (E_1)_u, (E_1)_v, (E_1)_{uv} \right]=0$. Moreover, 
$$
\left[ (E_2)_u, (E_2)_v, (E_2)_{uu} \right]=\frac{(\lambda_2)_v(\lambda_2-\lambda_1)^2}{(\lambda_2)^4}\left[f_u,f_{uu}\right],
$$
$$
\left[ (E_2)_u, (E_2)_v, (E_2)_{vv} \right]=\frac{(\lambda_2)_v^2(\lambda_2-\lambda_1)}{(\lambda_2)^4}\left[f_u,f_{v}\right],
$$
and $\left[ (E_2)_u, (E_2)_v, (E_2)_{uv} \right]=0$.
\end{lemma}

\begin{proof}
Immediate from Lemma \ref{Lemma:SingularFocalSets}.
\end{proof}

\begin{Proposition}\label{prop:Metric}
Consider a regular point of the line congruence. Then:
\begin{enumerate}
\item
The projective metric metric in the focal set $E_1$ is conformal to 
$$
ds_1^2=(\lambda_1)_u [f_u, f_v]du^2+(\lambda_2-\lambda_1)[f_v, f_{vv}] dv^2.
$$

\item
The projective metric in the focal set $E_2$ is conformal to 
$$
ds_2^2=(\lambda_1-\lambda_2)[f_u,f_{uu}] du^2 -(\lambda_2)_v [f_u,f_v]dv^2.
$$
\end{enumerate}
\end{Proposition}

\begin{proof}
Immediate from Lemma \ref{Lemma:FocalMetric}. 
\end{proof}

\subsection{Ridge and subparabolic sets}

From now on we shall consider coordinates $(x,y)\in\mathcal{D}$ such that the line congruence is defined by Equation \eqref{eq:Congruence} with $f(x,y)=(x,y,0)$.  When there is no risk of confusion, we write simply $f(x,y)=(x,y)$. Denote by $(\lambda_i,w_i)$, $i=1,2$, the eigenvalues and eigenvectors of the shape operator. Then
$$
\lambda_{1}=\frac{1}{2}\left( a+d+\sqrt{\delta} \right) \ \ {\rm and} \ \ \lambda_{2}=\frac{1}{2}\left( a+d-\sqrt{\delta}. \right),
$$
Moreover, if $c\neq 0$, we have the eigenvectors
$$
w_1=\left(  a-d+\sqrt{\delta}, 2c\right), \ \ w_2=\left(  a-d-\sqrt{\delta}, 2c\right).
$$
Since
$$
f_{w_1}=\left(  a-d+\sqrt{\delta}, 2c\right), \ \ f_{w_1}=\left(  a-d-\sqrt{\delta}, 2c\right),
$$
we obtain
$$
\left[ f_{w_1}, f_{w_2} \right]= 4c\sqrt{\delta}.
$$
	
A point $(x,y)$ belongs to the $i$-ridge set if the $i$-focal set is singular at the corresponding point. By Lemma \ref{Lemma:SingularFocalSets}, it is characterized by the condition $(\lambda_i)_{w_i}=0$. The following lemma appear in \cite{Craizer-Garcia-1}:
	
\begin{lemma}\label{lemma:Ridge}
Assume that $f(x,y)=(x,y)$ and $c\neq 0$. Outside the discriminant set, 
$$
(\lambda_1)_{w_1}=\frac{1}{\sqrt{\delta}}\left( A\sqrt{\delta}+B \right),\ \ \ (\lambda_2)_{w_2}=\frac{1}{\sqrt{\delta}}\left( A\sqrt{\delta}-B \right),
$$
where
$$
A=(a-d)a_{x}+cd_{y}+2ca_{y}+bc_{x},
$$
$$
B=(a-d)^2a_{x}+(a-d)(2ca_{y}-cd_{y}+bc_{x})+2c(2bd_{x}+ba_{x}+cb_{y}).
$$
\end{lemma}

As a consequence of the Lemma \ref{lemma:Ridge}, the $i$-ridge set is defined outside the discriminant set by the Equation $R_i=0$, where
\begin{equation*}\label{eq:DefineRi}
R_i=A\sqrt{\delta}+(-1)^{i-1}B, \ \ \ i=1,2.
\end{equation*}
The union of both ridges is given by $R=0$, where
\begin{equation}\label{eq:DefineR}
R=R_1\cdot R_2=A^2\delta-B^2. 
\end{equation}

\begin{remark}
In \cite[Sec.3]{Agafonov1}, the authors presented the following conditions for a system to have all points belonging to both ridge sets:
$$
(d-a)a_x-2ca_y-cd_y-bc_x=0, \ \ (a-d)d_y-2bd_x-ba_x-cb_y=0.
$$
One can easily check that these conditions is equivalent to $A=B=0$ in the above lemma.
\end{remark}

A point $(x,y)$ belongs to the $i$-subparabolic set if the $i$-focal set is parabolic at the corresponding point. It is characterized by the condition $[f_{w_j},f_{w_jw_j}]=0$, $j\neq i$. The following lemma can be found in \cite[Sec.2.3]{Craizer-Garcia-2}:

\begin{lemma}\label{lemma:SubP}
Assume that $c\neq 0$. Outside the discriminant set, 
$$
 \left[ f_{w_1}, f_{w_1w_1} \right]=\frac{4}{\sqrt{\delta}} \left( \bar{A}\sqrt{\delta}+\bar{B}\right), \ \ \ \left[ f_{w_2}, f_{w_2w_2} \right]=\frac{4}{\sqrt{\delta}} \left( \bar{A}\sqrt{\delta}-\bar{B}\right), 
$$
where
$$
\bar{A}=c_x((a-d)^2+bc) +c(d-a)(a_x-2d_x)+c^2(d_y-2a_y),
$$
$$
\bar{B}= -c((a - d)^2 + 2bc)(a_x - 2d_x) + ((a - d)^2 + 3bc)(a - d)c_x - c^2(a - d)(2a_y - d_y) - 2c^3b_y.
$$
\end{lemma}

As a consequence of the above lemma, the $i$-subparabolic set is defined outside the discriminant set by the Equation $S_i=0$, where
\begin{equation*}\label{eq:DefineSi}
S_i=\bar{A}\sqrt{\delta}+(-1)^{i}\bar{B}, \ \ \ i=1,2.
\end{equation*}
The union of both subparabolic sets is given by $S=0$, where
\begin{equation}\label{eq:DefineS}
S=S_1\cdot S_2=\bar{A}^2\delta-\bar{B}^2. 
\end{equation}

\begin{remark}
In \cite[Sec.3]{Agafonov1}, the authors presented the following conditions for a system to have all points belonging to both subparabolic sets:
$$
(a-d)bc_x+bc(2d_x-a_x)-c^2b_y=0, \ \ (d-a)cb_y+bc(2a_y-d_y)-b^2c_x=0.
$$
One can easily check that these conditions is equivalent to $\bar{A}=\bar{B}=0$ in the above lemma.
\end{remark}

\subsection{A formula defining $W$-congruences}

Consider a line congruence defined by Equation \eqref{eq:Congruence} with $f=(x,y)$. 
The following proposition is a direct consequence of Proposition \ref{prop:Metric}:

\begin{Proposition}
At a regular point, the line congruence is Weingarten if and only if 
\begin{equation}\label{eq:CharW}
\left|
\begin{array}{cc}
(\lambda_1)_{w_1}[f_{w_1},f_{w_2}]  & (\lambda_1-\lambda_2)[f_{w_2},f_{w_2w_2}]  \\
(\lambda_1-\lambda_2) [f_{w_1},f_{w_1w_1}]  & (\lambda_2)_{w_2}[f_{w_1},f_{w_2}] 
\end{array}
\right|=0.
\end{equation}
\end{Proposition}

Consider
\begin{equation}\label{eq:DefineH}
H=c^2R-S,
\end{equation}
where $R$ and $S$ are given by Equations \eqref{eq:DefineR} and \eqref{eq:DefineS}, respectively. 

\begin{corollary}\label{prop:CharH}
Assume $c\neq 0$ at a regular point of the line congruence. Then the line congruence is Weingarten if and only if $H=0$.
\end{corollary}

\begin{proof}
From the above proposition, the condition for $W$-congruence is given by
$$
(\lambda_1)_{w_1}(\lambda_2)_{w_2}[f_{w_1},f_{w_2}]^2-(\lambda_1-\lambda_2)^2[f_{w_1},f_{w_1w_1}][f_{w_2},f_{w_2w_2}]=0.
$$
Assuming that $c\neq 0$, Lemmas \ref{lemma:Ridge} and \ref{lemma:SubP} imply that this equation
can be written as 
$$
c^2(A^2\delta-B^2)-(\bar{A}^2\delta-\bar{B}^2)=0,
$$
thus proving the proposition.
\end{proof}

\begin{Proposition}\label{prop:HW}
Assume that $c\neq 0$ at a regular point of the line congruence and let $H$ be given by Equation \eqref{eq:DefineH}. Then
$$
H=4c^3\delta W,
$$
where $W$ is given by Formula \eqref{eq:DefineW}.
\end{Proposition}

\begin{proof}
Straightforward calculations.
\end{proof}

\begin{corollary}\label{cor:W=0}
At a regular point, the line congruence is Weingarten if and only if $W=0$.
\end{corollary}

\begin{proof}
If $c(0,0)\neq 0$, then the result is a direct consequence of Proposition \ref{prop:HW}. If $b(0,0)\neq 0$, by symmetry, the same result holds. It is not difficult to see that the result holds also if $b\neq 0$ or $c\neq 0$ for a sequence converging to the origin. Thus we have only to consider the case $b=c=0$ in a neighborhood of the origin.
In this case, $a=a(x)$ and $d=d(y)$ and we obtain $W=(d-a)a_xd_y$,  $\bar{A}=\bar{B}=0$. Observe also that the eigenvectors of the shape operator are $(1,0)$ and $(0,1)$ at all points, and so $[f_x,f_{xx}]=[f_y,f_{yy}]=0$. Hence Equation \eqref{eq:CharW} implies that the congruence is Weingarten if and only if $a_xd_y=0$, thus proving the corollary.
\end{proof}

At singular points, we can use the Equation $W=0$ as a definition of the Weingarten condition. 
Taking into account this definition and Corollary \ref{cor:W=0}, we conclude that the Weingarten condition is equivalent to the equation $W=0$. 

\begin{remark}
Since
$H=c^2R-S$, we conclude that, for a $W$-congruence, an $i$-ridge point is also a $j$-subparabolic point, $j=1$ or $j=2$. 
\end{remark}

\begin{remark}
To each conservation law in two variables, we can associate a line congruence. An important class of conservation laws is the {\it Temple class} (\cite{Agafonov1},\cite{Agafonov2}). In terms of line congruences, the Temple class of type $1$ corresponds to the conditions $S_1=S_2=0$ at all points, the Temple class of type $2$ corresponds to $R_1=R_2=0$ at all points, while the Temple class of type $3$ corresponds to $R_i=S_i=0$, at all points. It is clear from the above discussion that only type $3$ Temple congruences are $W$-congruences. 
\end{remark}

\subsection{Normal congruences} 

Following \cite[p.118]{Sasaki}, a normal congruence is Weingarten if and only if the principal radii $r_i$, $i=1,2$, satisfy the relation
\begin{equation}\label{eq:Radii}
\left|
\begin{array}{cc}
(r_1)_x   &  (r_2)_x \\
(r_1)_y   &  (r_2)_y
\end{array}
\right|=0.
\end{equation}
We shall now verify that this condition coincides with $W=0$ in the case of a normal congruence.

Consider the normal congruence of the graph of a function 
$$
g(x,y)=\frac{1}{2}\left( g_{20}x^2+g_{02}y^2\right)+\frac{1}{6}\left(g_{30}x^3+3g_{21}x^2y+3g_{12}xy^2+g_{03}y^3\right)+O(4).
$$
with $g_{20}\neq g_{02}$. It is not difficult to verify that, at the origin, 
$$
k_1=g_{20}, \ \ k_2=g_{02},\ \ (k_1)_x=g_{30}, \ \ (k_1)_y=g_{21}, \ \ (k_2)_x=g_{12}, \ \ (k_2)_y=g_{03}.
$$
Thus condition \eqref{eq:Radii} for $W$-congruence at a non-umbilical point can be written as 
\begin{equation}\label{eq:RadiiExample1}
g_{30}g_{03}-g_{21}g_{12}=0.
\end{equation}

The lines of the congruence are given by
$$
(x,y,g(x,y))+\lambda(-g_x,-g_y,1).
$$
Taking $\lambda=-g(x,y)$, we obtain
$$
(X,Y)=\left(x+g_xg,y+g_yg\right), \ \ (\xi_1,\xi_2)=(-g_x,-g_y).
$$
Now straightforward calculations show that, at $(x,y)=(0,0)$, 
$$
a=g_{20},\ a_X=g_{30},\ a_Y=g_{21}, \ \ \ d=g_{02},\ d_X=g_{12},\ d_Y=g_{03},
$$
$$
b=0,\ b_X=g_{21},\ b_Y=g_{12}, \ \ \ c=0,\ c_X=g_{21},\ c_Y=g_{12}.
$$
Thus the condition $W=0$ at a non-umbilical point is 
$$
g_{30}g_{03}-g_{21}g_{12}=0,
$$
which coincides with formula \eqref{eq:RadiiExample1}.

\section{Jet space at non-degenerate umbilical points}

We shall consider from now on analytical germs $\xi=(\xi^1,\xi^2)$ at the origin. We write
\begin{equation*}\label{Expansion1}
\xi^1(x,y)=\sum_{n=1}^{\infty}\left( \frac{1}{n!}\sum_{k=0}^n {{n}\choose{k}} p_{n-k,k}x^{n-k}y^k\right),
\end{equation*}
\begin{equation*}\label{Expansion2}
\xi^2(x,y)=\sum_{n=1}^{\infty}\left( \frac{1}{n!}\sum_{k=0}^n {{n}\choose{k}} q_{n-k,k}x^{n-k}y^k\right).
\end{equation*}
Given a smooth germ of function $g:(\mathbb{R}^2,(0,0))\to\mathbb{R}$, we shall denote by 
$g_{n-k,k}$ the $n$-order derivative of the germ $g$, $(n-k)$ times with respect to the first variable and $k$ times with respect to the second variable. When there is no risk of misunderstanding, we may also drop the comma.

\subsection{Jet space at a point of an analytical $W$-congruence}

We now describe the $C^{\omega}$-topology in a neighborhood of a point, that we assume to be the origin. 
Since $\xi$ is analytic, the condition $W=0$ in a neighborhood of the origin is equivalent to Equations \eqref{eq:WDerivatives},
for any $n\in\mathbb{N}$, $0\leq k\leq n$. Equations \eqref{eq:WDerivatives} determine relations between the coefficients $p_{n-k,k}$ and $q_{n-k,k}$. Some of these coefficients are free, and the other are obtained as functions of the free ones. We shall consider perturbations of the analytical germ $\xi$ by changing slightly a finite number of the free coefficients.

\begin{lemma}
Consider an analytical $W$-congruence defined by the direction vector $\xi$. Denote by $\tilde\xi$ a perturbation of $\xi$ obtained by changing a finite number of free coefficients. Then the convergence radius of $\tilde\xi$ is at least $\min(r,\tfrac{1}{3})$, where $r$ denotes the convergence radius of the initial $\xi$. 
\end{lemma}

\begin{proof}
The non-free coefficients of order $n$ are obtained as the order $n$ derivatives of at most six parcels, each of them a product of three factors. Thus the derivative of order $n$ is obtained by a sum of at most $6\cdot 3^n$ parcels, each parcel with three factors. We conclude that 
the perturbation of the non-free coefficients of order $n$ are bounded by $C\cdot 3^n$, for some constant $C$, which proves the lemma.
\end{proof}

Fix a neighborhood $U_0$ of the origin where $\xi$ is given by its Taylor expansion at the origin and denote by $\mathcal{W}_0$ the space of line congruences defined in $U_0$ with the $C^{\omega}$-topology.
From the above lemma, a basis for the open sets of  $\mathcal{W}_0$ consists of the open sets of the Euclidean space of free coefficients $p_{n-k,k}$ and $q_{n-k,k}$, with $n\leq N$, $N\in\mathbb{N}$.   

\subsection{Non-degenerate umbilical points}

In this section, we shall consider umbilical points of $W$-congruences.  We say that an umbilical point of a $W$-congruence is {\it isolated} if it admits a neighborhood $U$ where it is the only umbilical point.
An isolated umbilical point is {\it stable} if it admits a neighborhood $U$ such that, for some neighborhood $\mathcal{U}$ of the $W$-congruence in the $C^{\omega}$-topology, every $W$-congruence in $\mathcal{U}$ admits a unique umbilical point in $U$.

We need some conditions in order to get $W$-stable umbilical points. At an umbilical point, we have that $p_{10}=q_{01}$ and $p_{01}=q_{10}=0$. Consider the following determinants:
\begin{equation*}
\Omega_{1}=\left|
\begin{array}{cc}
\xi^1_{xx} & \xi^1_{xy} \\
\xi^2_{xx} & \xi^2_{xy}
\end{array}
\right|
, \ \ \ 
\Omega_{2}=\left|
\begin{array}{cc}
\xi^1_{xy} & \xi^1_{yy} \\
\xi^2_{xy} & \xi^2_{yy}
\end{array}
\right|
, \ \ \ 
\Omega_{3}=\left|
\begin{array}{cc}
\xi^1_{xx} & \xi^1_{yy} \\
\xi^2_{xx} & \xi^2_{yy}
\end{array}
\right|,
\end{equation*}
and 
\begin{equation*}
J_{1}=\left|
\begin{array}{cc}
\xi^1_{xx}-\xi^2_{xy} & \xi^1_{xy}-\xi^2_{yy} \\
\xi^2_{xx} & \xi^2_{xy}
\end{array}
\right|,\ 
J_{2}=\left|
\begin{array}{cc}
\xi^1_{xx}-\xi^2_{xy} & \xi^1_{xy}-\xi^2_{yy} \\
\xi^1_{xy} & \xi^1_{yy}
\end{array}
\right|,\ 
J_{3}=\left|
\begin{array}{cc}
\xi^1_{xy} & \xi^1_{yy} \\
\xi^2_{xx} & \xi^2_{xy}
\end{array}
\right|.
\end{equation*}
We say that an umbilical point is {\it non-degenerate} if $\Omega_i\neq 0$ and $J_i\neq 0$, for some $i=1,2,3$.

\begin{Proposition}
A non-degenerate umbilical point is isolated and stable.
\end{Proposition}

\begin{proof}
We prove the proposition for $i=2$, the other cases being similar. Let $\xi$ be a directional vector field with $\Omega_2\neq 0$ and $J_2\neq 0$ at the origin. 
Consider a $W$-congruence $\eta=(\eta^1,\eta^2)$ $C^{\omega}$-close to $\xi$. 
We first solve the equations 
\begin{equation}\label{eq:J2}
\begin{cases}
\eta^1_x-\eta^2_y=0\\
\eta^1_y=0
\end{cases}
\end{equation}
The jacobian matrix is 
$$
\left[
\begin{array}{cc}
\eta^1_{xx}-\eta^2_{xy} & \eta^1_{xy}-\eta^2_{yy}\\
\eta^1_{xy} & \eta^1_{yy}
\end{array}
\right],
$$
which is close to $J_2$ and hence non-zero, which implies that the Equations \eqref{eq:J2} admit a unique solution $x=x_0$, $y=y_0$ close to the origin.
Now the Equation $W=0$
implies that 
$\eta_x^2\Omega_2(\eta)=0$
at $(x_0,y_0)$. Since $\Omega_2(\eta)$ is close to $\Omega_2$ and hence non-zero, we conclude that $\eta_x^2(x_0,y_0)=0$ and so $(x_0,y_0)$ is umbilical for $\eta$.
\end{proof}

\begin{remark}
At an umbilical point, the value of $p_{10}=q_{01}$ is not relevant for the condition $W=0$. It is also not relevant 
for the discriminant $J$. Thus, when convenient, we may assume that $p_{10}=q_{01}=0$.
\end{remark}

\begin{lemma}
Assume that the origin is a non-degenerate umbilical point. By linear changes in the line congruence, we may assume that
\begin{equation}\label{Conditions2jet}
p_{02}=p_{20}=q_{11}=0,\ \ p_{11}\neq 0, \ \ q_{02}\neq 0.
\end{equation}
Under these conditions, the non-degenerate hypothesis implies that $p_{11}\neq q_{02}$. Moreover, if $p_{11}+q_{02}\neq 0$ then $q_{20}=0$.
\end{lemma}

\begin{proof}
By the above remark, we may assume $p_{10}=q_{01}=0$. Denote the $2$-jet of $\xi^i$ by $Q^i$. 
If both $Q^1$ and $Q^2$ are definite, by a linear transformation we may assume that $Q^2=x^2+y^2$,
and then, by a rotation, $Q^1=\tfrac{1}{2}\left(p_{20}x^2+p_{02}y^2  \right)$. Now the conditions $W_x=W_y=0$ at the origin gives that $p_{20}=p_{02}$, which contradicts the non-degeneracy hypothesis.

Assume now that both forms are degenerate non-zero. We may assume $Q^2=\tfrac{1}{2}q_{02}y^2$, $q_{02}\neq 0$. Then, by the linear transformation, we may assume that $p_{02}=0$. Since $Q^1$ is degenerate, we conclude that 
$p_{11}=0$. Now the conditions $W_x=W_y=0$ at the origin implies that $p_{20}q_{02}=0$, again contradicting the non-degeneracy hypothesis. 

We conclude that one of the quadratic forms, say $Q^1$, is indefinite. We may assume that its asymptotic directions are the coordinate axes, which is equivalent to say that $Q^1=p_{11}xy$, $p_{11}\neq 0$. Then by a linear transformation, we obtain that $q_{11}=0$. The condition $W_x=0$ implies that
$q_{20}(p_{11}+q_{02})=0$. If $q_{02}=0$, then $q_{20}=0$, contradicting the non-degenerate condition. We conclude that $q_{02}\neq 0$, thus proving the lemma. 
\end{proof}

From now on, we shall assume that conditions \eqref{Conditions2jet} hold and that $p_{11}\neq q_{02}$. Denote 
$$
m=-\tfrac{q_{02}}{p_{11}}. 
$$
From the above lemma, if $m\neq 1$, then necessarily $q_{20}=0$. In this case, the direction $(1,0)$ is asymptotic for both $Q^1$ and $Q^2$. In next section this direction will appear as the kernel of the hessian of the discriminant at some stable umbilical points.

\subsection{Jet space at a non-degenerate umbilical point}

Denote 
\begin{equation}\label{eq:Defineh}
h=\frac{ (\xi^1_x-\xi^2_y)(\xi^2_{xx}\xi^1_{yy}-\xi^1_{xx}\xi^2_{yy})+2\xi^1_y(\xi^1_{xx}\xi^2_{xy}-\xi^2_{xx}\xi^1_{xy})}{2(\xi^1_{xy}\xi^2_{yy}-\xi^1_{yy}\xi^2_{xy})}
\end{equation}

Then, at a non-degenerate umbilical point, $h=\xi^2_x$. In order to understand the jet space at an umbilical point, we need the following two lemmas:

\begin{lemma}\label{Lemma:JetSpace1}
For $0\leq k\leq n$, 
$$
h_{n-k,k}= \tfrac{1}{m}(n-k)q_{n+1-k,k}-\frac{q_{20}}{mq_{02}}(n-k)q_{n-1-k,k+2}+\Gamma(n+1,n),
$$
where $\Gamma(n+1,n)$ denotes terms that involve derivatives of order at most $(n+1)$ in $\xi^1$ and at most $n$ in $\xi^2$.
\end{lemma}

\begin{proof}
From Equation \eqref{eq:Defineh}, the only terms involving derivatives of degree $n+1$ in $\xi^2$ that appear in $h_{n-k,k}$ are
$-\frac{1}{q_{02}}(\xi^1_y\xi^2_{xx})_{n-k,k}$  and $-2q_{20}\left(\frac{\xi^1_y}{\xi^2_{yy}}\right)_{n-k,k}$. Moreover, the terms involving derivatives of degree $n+1$ in $\xi^2$ that appear are exactly
$$
-\frac{p_{11}}{q_{02}}(n-k)q_{n+1-k,k},\ \ \ \frac{p_{11}q_{20}}{q_{02}^2}(n-k)q_{n-1-k,k+2},
$$
thus proving the lemma. 
\end{proof}

\begin{corollary}\label{Cor:JetSpace1}
Assume that $m\neq 1$ and $n-k\neq m$. Then the equation $W_{n-k,k}=0$ can be written in the form
$$
q_{n+1-k,k}=\Gamma(n+1,n), 
$$
where $\Gamma(n+1,n)$ denotes terms that involve derivatives of order at most $(n+1)$ in $\xi^1$ and at most $n$ in $\xi^2$.
\end{corollary}

\begin{proof}
Since $m\neq 1$, we have that $q_{20}=0$. Moreover, since $h_{n-k,k}=q_{n+1-k,k}$ and $n-k\neq m$, Lemma \ref{Lemma:JetSpace1} says that, in the equation $W_{n-k,k}$ can be written in the the above form. 
\end{proof}

\begin{lemma}\label{Lemma:JetSpace2}
For , $m>1$,  $n-k=m$,  $k\geq 1$,
$$
h_{m,k}=q_{m+1,k} -\tfrac{k(m+1)}{2} p_{m+2,k-1}+\Gamma(n,n),
$$
where $\Gamma(n,n)$ denotes terms that involve derivatives of order at most $n$ in $\xi^1$ and at most $n$ in $\xi^2$.
\end{lemma}

\begin{proof}
By the above lemma, since $m\neq 1$, the only term of order $n+1$ in $\xi^2$ is $q_{m+1,k}$. On the other hand, the only term 
of order $(n+1)$ in $\xi^1$ is
$-\frac{1}{2p_{11} }\left((\xi^1_x-\xi^2_y)\xi^1_{xx}\right)_{m,k}$. Moreover, the term involving derivatives of degree $n+1$ in $\xi^2$ that appear is exactly
$$
-\frac{1}{2p_{11}}(p_{11}-q_{02})kp_{m+2,k-1}
$$
thus proving the lemma. 
\end{proof}

The following corollary is straightforward:

\begin{corollary}\label{Cor:JetSpace2}
Assume that $m\in\mathbb{N}$, $m>1$, $n-k=m$, $k\geq 1$. Then the equation $W_{n-k,k}=0$ can be written in the form
$$
p_{n-k+2,k-1}=\Gamma(n,n),
$$
where $\Gamma(n,n)$ denotes terms that involve derivatives of order at most $n$ in $\xi^1$ and at most $n$ in $\xi^2$.
\end{corollary}

These relations show that, for each equation obtained as a derivative of $W$ of order $n$, we can write some coefficient of order $(n+1)$ as a function of other coefficients of order at most $(n+1)$. There is one exception to this behavior, the equation $W_{m0}=0$. In fact, the above lemmas are not conclusive when $k=0$ and $n=m$. Nevertheless, in all the examples below, the equation $W_{m0}=0$ holds automatically, and so does not determines a new relation among the coefficients.

Since $W_{m0}=0$ at the origin in all examples of this paper, we expect that this equation holds in general, and so we formulate the following conjecture: For any $m\in\mathbb{N}$,
\begin{equation}\label{Conjecture1}
W_{m0}=0.
\end{equation}

\medskip\noindent
If this conjecture is true, the space of jets at an umbilical point is completely determined by Lemmas \ref{Lemma:JetSpace1} and \ref{Lemma:JetSpace2}. 
Since we have verified the conjecture for $1\leq m\leq 4$, we can describe completely the jet space in these cases. The same is true if $m\not\in\mathbb{N}$, since in this case the conjecture is vacuously true. 

\section{Stable umbilics of type $A_m$, $m=1,2,3,4$.}

\subsection{Type of singularity of a germ}

Consider a germ $g$ of a real function at a singular point, which we assume to be the origin. We say that $g$ is of type $A_m$ if it is possible, by a change of variables, to write it as 
$$
g=y^2+x^{m+1}, 
$$
(see  \cite[Ch. 4]{Gibson}). If, by a change of variables, we can write $g(x,y)=y^2$, we say that $g$ is of type $A_{\infty}$. 

Assume that the second jet of $g$ is singular with the $x$-axis as kernel, which is equivalent to say that
$$
g(x,y)=\tfrac{1}{2}g_{02}y^2+O(3).
$$
The following criteria for $A_m$ singularity, $m=2,3,4$, is well-known:

\begin{Proposition}
Consider a germ $g$ at a singular point as above:
\begin{enumerate}
\item
$g$ is of type $A_2$ if and only if $g_{30}\neq 0$.

\item
Assuming $g$ is not $A_2$, $g$ is of type $A_3$ if and only if
\begin{equation}\label{eq:T3}
T_3=g_{40}g_{02}-3g_{21}^2\neq 0.
\end{equation}

\item
Assuming $g$ is not $A_3$, $g$ is of type $A_4$ if and only if 
\begin{equation}\label{eq:T4}
T_4=g_{50}g_{02}^2-10g_{31}g_{21}g_{02}+15g_{12}g_{21}^2\neq 0.
\end{equation}
\end{enumerate}
\end{Proposition}

\subsection{General approach}
We shall consider in this section non-degenerate umbilical points with $m=1,2,3,4$. We shall describe the jet space in each case, calculate the $(m+1)$-jet of the discriminant function $\delta$ and prove that it is of type $A_m$. Moreover, we shall verify that at least one derivative of $R$ is non-zero, thus concluding that the umbilical point is a limit of regular points. 

In case $m=1$, Lemma \ref{Lemma:JetSpace1} with $n=2$, $k=0,2$ says that we can express $q_{30}$ and $q_{12}$ in terms of coefficients of $\xi^1$ up to order $3$ and $\xi^2$ up to order $2$. 
For $n=2$ and $k=1$,  the same lemma says that we can express $q_{03}$ in terms of coefficients of $\xi^1$ up to order $3$ and $\xi^2$ up to order $2$. 
In general, Lemma \ref{Lemma:JetSpace1} says that  all coefficients of $\xi^1$ are free and also $(q_{2,k})$, $k\geq 0$ are free. The other coefficients can be obtained as functions of them. 

In case $m=2$, Corollary \ref{Cor:JetSpace1} with $n=2$, $k=1,2$ says that we can express $q_{21}$ and $q_{12}$ in terms of coefficients of $\xi^1$ up to order $3$ and $\xi^2$ up to order $2$. 
Similarly, Corollary \ref{Cor:JetSpace1} with $n=3$, $k=0,2,3$ says that we can express $q_{40}$, $q_{22}$ and $q_{13}$ in terms of coefficients of $\xi^1$ up to order $4$ and $\xi^2$ up to order $3$. 
Moreover, Corollary \ref{Cor:JetSpace2} with $n=3$, $k=1$ says that we can express $p_{40}$ in terms of coefficients of $\xi^1$ up to order $3$ and $\xi^2$ up to order $3$.
In general, the coefficients $q_{3,k}$, $k\geq 0$ are free, together with the coefficients
$q_{0,k}$, $k\geq 2$ and $p_{l,k}$, $l\neq 4$. On the other hand, $q_{l,k}$, $l\neq 0,3$ and $p_{4,k}$, $k\geq 0$ are obtained as functions of the other coefficients. 

In case $m=3$, Corollary \ref{Cor:JetSpace1} with $n=2$, $k=0,1,2$ says that we can express $q_{30}$, $q_{21}$ and $q_{12}$ in terms of coefficients of $\xi^1$ up to order $3$ and $\xi^2$ up to order $2$. The same corollary with $n=3$, $k=1,2,3$, says that we can express $q_{31}$, $q_{22}$ and $q_{13}$ in terms of coefficients of $\xi^1$ up to order $4$ and $\xi^2$ up to order $3$. For $n=4$, $k=0,2,3,4$, we can express $q_{50}$, $q_{32}$, $q_{23}$ and $q_{14}$ in terms of coefficients of $\xi^1$ up to order $5$ and $\xi^2$ up to order $4$. The free coefficients are $q_{4,k}$.
Corollary \ref{Cor:JetSpace2} with $n=4$, $k=1$, says that we can express $p_{50}$ in terms of coefficients of $\xi^1$ up to order $4$ and $\xi^2$ up to order $4$. In general, the coefficients $p_{5k}$ are obtained as functions of the other coefficients.

In case $m=4$, Corollary \ref{Cor:JetSpace1} with $n=2$, $k=0,1,2$ says that we can express $q_{30}$, $q_{21}$ and $q_{12}$ in terms of coefficients of $\xi^1$ up to order $3$ and $\xi^2$ up to order $2$. The same corollary with $n=3$, $k=0,1,2,3$ says that we can express $q_{40}$, $q_{31}$, $q_{22}$ and $q_{13}$ in terms of coefficients of $\xi^1$ up to order $4$ and $\xi^2$ up to order $3$. For $n=4$, $k=1,2,3,4$, we can express $q_{41}$, $q_{32}$, $q_{23}$ and $q_{14}$ in terms of coefficients of $\xi^1$ up to order $5$ and $\xi^2$ up to order $4$. For $n=5$, $k=0,2,3,4,5$, we can express $q_{60}$, $q_{42}$, $q_{33}$, $q_{24}$ and $q_{15}$ in terms of coefficients of $\xi^1$ up to order $6$ and $\xi^2$ up to order $5$. The free coefficients are $q_{5,k}$.
Corollary \ref{Cor:JetSpace2} with $n=5$, $k=1$, says that we can express $p_{60}$ in terms of coefficients of $\xi^1$ up to order $5$ and $\xi^2$ up to order $5$. In general, the coefficients $p_{6k}$, $k\geq 0$, are obtained as functions of the other coefficients.

\subsection{Umbilical points of type $A_m$, $m=1,2,3,4$}

\subsubsection{Umbilical points of Morse type}

Assume that $m=1$. Since 
\begin{equation}\label{eq:W10}
W_{10}=2q_{20}p_{11}^2(1-m),
\end{equation}
is automatically zero, we have $q_{20}\neq 0$ for an open and dense subset of $\mathcal{W}_0$. 
The $2$-jet $\delta_2$ of the discriminant $\delta$ is given by 
\begin{equation*}
\delta_2(x,y)=4p_{11}\left(  p_{11}y^2+q_{20}x^2 \right).
\end{equation*}
Since $p_{11}q_{20}\neq 0$, the umbilical point is of Morse type. 
We have thus proved the following proposition:
\begin{Proposition}
Assume that $m=1$. Then, for an open and dense subset of $\mathcal{W}_0$, the umbilical point is Morse. 
\end{Proposition}

To be sure that the umbilical point is a limit of regular points, we have calculated $R_{22}$ at the origin, obtaining $R_{22}=-192p_{11}^4q_{20}^2\neq 0$.

\subsubsection{Umbilical points of type $A_2$}

Assume that $m=2$. 
In this case, Equation \eqref{eq:W10} implies that $q_{20}=0$, and since
\begin{equation}\label{eq:W20}
W_{20}=2p_{11}q_{30}(2p_{11}+q_{02})=2(2-m)p_{11}^2q_{30}.
\end{equation}
is automatically zero, we have that $q_{30}\neq 0$ for an open and dense set in $\mathcal{W}_0$.  
The second jet of $\delta$ is given by
\begin{equation}\label{eq:delta02}
\delta_2(x,y)=9p_{11}^2y^2.
\end{equation}
We have also that
\begin{equation*}\label{eq:delta30}
\delta_{30}=12p_{11}q_{30},
\end{equation*}
and so we obtain that the singularity of the discriminant is of type $A_2$. We have thus proved the following proposition:

\begin{Proposition}
Assume that $m=2$. Then, for an open and dense subset of $\mathcal{W}_0$, the umbilical point is of type $A_2$. 
\end{Proposition}

To be sure that the umbilical point is a limit of regular points, we have calculated $R_{42}$ at the origin, obtaining $R_{42}=-3456p_{11}^4q_{30}^2\neq 0$.

\subsubsection{Umbilical points of type $A_3$}

Assume that $m=3$. From Equation \eqref{eq:W20}, we obtain that necessarily $q_{30}=0$.  
Since
\begin{equation}\label{eq:W30}
W_{30}=-\frac{p_{11}(m-3)\left(4p_{11}q_{40}(m-1)+3mp_{30}^2(m+2) \right)}{2(m-1)}
\end{equation}
is automatically zero, the condition
\begin{equation}\label{eq:q40}
8p_{11}q_{40}+45p_{30}^2 =0,
\end{equation}
does not hold for an open and dense set in $\mathcal{W}_0$. Considering the equations $W_{11}=W_{02}=0$, we obtain that
$q_{21}=-3p_{30}$ and $q_{12}=-4p_{21}$.

For the discriminant function $\delta$, 
straightforward calculations show that
$$
T_3=-\tfrac{64}{3} p_{11}^2 \left(45p_{30}^2+8p_{11}q_{40}\right),
$$
where $T_3$ is given by Equation  \eqref{eq:T3}.
Since Equation \eqref{eq:q40} does not hold, we conclude that the singularity of the discriminant is of type $A_3$.
We have thus proved the following proposition:

\begin{Proposition}
Assume that $m=3$. Then, for an open and dense subset of $\mathcal{W}_0$, the umbilical point is of type $A_3$. 
\end{Proposition}

To be sure that the umbilical point is a limit of regular points, we have calculated $R_{24}$ at the origin, obtaining
$R_{24}=-138240p_{11}^4p_{30}^2\neq 0$. 

\subsubsection{Umbilical points of type $A_4$}

Assume that $m=4$. 
The condition $W_{40}=0$ can be written in the form $W_{40}=(m-4)B$, where
$$
B=(5m^3+3m^2-14m)p_{11}p_{30}p_{40}-(5m^3+12m^2-4m)p_{21}p_{30}^2+(2m^2-6m+4)p_{11}^2q_{50}.
$$
Thus we obtain that the equation
\begin{equation}\label{eq:T4m4}
p_{11}^2q_{50}+26p_{11}p_{30}p_{40}-44p_{21}p_{30}^2=0
\end{equation}
does not hold for an open and dense set in $\mathcal{W}_0$. 

Considering the equations $W_{20}=W_{30}=0$ we obtain that $q_{30}=0$ and $q_{40}=-\tfrac{6p_{30}^2}{p_{11}}$.
From equations $W_{11}=W_{02}=0$, we obtain that
$q_{21}=-\tfrac{10}{3}p_{30}$ and $q_{12}=-5p_{21}$. Finally From equations $W_{30}=W_{12}=W_{21}=W_{03}=0$ we obtain that
$$
q_{31}=-\frac{10p_{11}p_{40}-9p_{21}p_{30}}{2p_{11}},
$$
$$
q_{22}=-\frac{60p_{11}p_{31}-84p_{12}p_{30}-18p_{21}^2-p_{30}q_{03}}{9p_{11}},
$$
$$
q_{13}=\frac{20p_{03}p_{30}-15p_{11}p_{22}+27p_{12}p_{21}+3p_{21}q_{03}}{2p_{11}}.
$$

For the discriminant function $\delta$, 
straightforward calculations show that
$$
T_4=-50000p_{11}^3(p_{11}^2q_{50}+26p_{11}p_{30}p_{40}-44p_{21}p_{30}^2),
$$
where $T_4$ is given by Equation  \eqref{eq:T4}.
Since Equation \eqref{eq:T4m4} does not hold, we conclude that the singularity is $A_4$. 
We have thus proved the following proposition:
\begin{Proposition}
Assume that $m=4$. Then, for an open and dense subset of $\mathcal{W}_0$, the umbilical point is of type $A_4$. 
\end{Proposition}

To be sure that the umbilical point is a limit of regular points, we have calculated $R_{24}$ at the origin, obtaining
$R_{24}=-320000p_{11}^4p_{30}^2\neq 0$.

\subsection{Some configurations of principal lines}

In this subsection, we describe the generic configuration of principal lines in cases $m=1$ and $m=2$.  The case $m=3$ is more complicated and we did not include in the paper. We could not find a pattern for the configurations, and we include these two cases just as examples. The blowing-up approach can be found in \cite{Davydov}. 

\subsubsection{Principal lines at Morse umbilical points}

The differential equation of the principal lines is given by 
\begin{equation}\label{EDBA1}
(p_{11}x+O(2))dy^2 + (2p_{11}y+O(2))dxdy -(q_{20}x+O(2))dx^2=0,
\end{equation}

\begin{Proposition}\label{prop:EDBA1}
Consider the differential equation \eqref{EDBA1}, with $q_{02}p_{11}>0$, case $A_1^{+}$, or $q_{02}p_{11}<0$, case $A_1^{-}$. Its solutions are as shown in Figure \ref{fig:UmbilicA1}. 
\end{Proposition}

\begin{figure}[htb]
\centering
\includegraphics[width=0.6\linewidth]{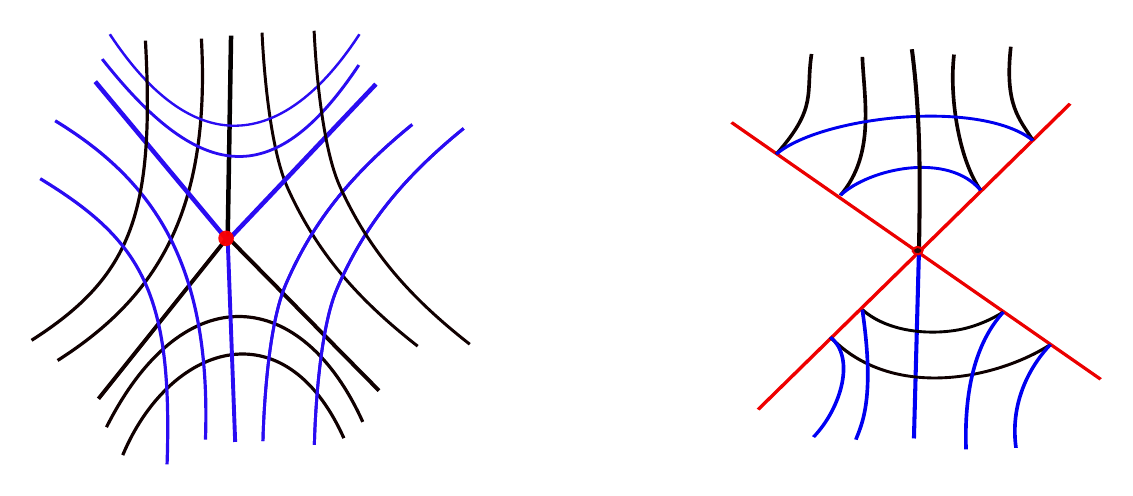}
 \caption{Left: Umbilical point with discriminant $A_1^{+}$, \ Right: Umbilical points with discriminant $A_1^{-}$}
\label{fig:UmbilicA1}
\end{figure}

\begin{proof}
Consider the blowing-up 
$$
H(u,v)=(u,uv), \ \ -\epsilon\leq u\leq \epsilon, \ \ -R\leq v\leq R.
$$
The pullback of the BDE can be written as
$$
\left(p_{11}u^2+O(3)\right)dv^2+\left(4p_{11}uv+O(2)\right)dudv+\left(3p_{11}v^2-q_{20}+O(1)\right)du^2=0,
$$
where here $O(k)$ means order $\geq k$ in $u$. 
In case $A_1^{+}$, the singular points of this BDE are 
$(u_1,v_1)=\left(0, \sqrt{\tfrac{q_{20}}{3p_{11}}}  \right)$ and  $(u_2,v_2)=\left(0, -\sqrt{\tfrac{q_{20}}{3p_{11}}}  \right)$.
In case $A_1^{-}$, the BDE has no singular points. Considering only the first order terms in $u$, the BDE can be factored as
$
\left(4p_{11}uvdv+(3p_{11}v^2-q_{20}+\alpha u)du\right)du=0,
$
where $\alpha$ is some real number that will not be relevant in our analysis. 
The integral lines of the first factor correspond to the integral lines of the vector field
$X=\left( 4p_{11}uv, -3p_{11}v^2+q_{20}-\alpha u \right).$ 
The determinant of $DX(0,v_i)=-24p_{11}^2v^2<0$, and thus both points are saddles. 

To analyze the behavior of the principal lines close to the $x$-axis,  consider the blowing-up $H_1(u,v)=(uv,v)$. Using a similar analysis as above, it follows that $(0,0)$ is a hyperbolic saddle. Taking these results into account, the proposition is proved.
\end{proof}

\subsubsection{Principal lines at $A_2$-umbilical points}

The differential equation of the principal lines is given by 
\begin{equation}\label{EDBA2}
(p_{11}x+O(2))dy^2 + (3p_{11}y+O(2))dxdy -\tfrac{1}{2}(q_{30}x^2+2q_{21}xy+q_{12}y^2+O(3))dx^2=0,
\end{equation}
where $q_{21}=-3p_{30}$, $q_{12}=-3p_{21}$. 

\begin{Proposition}\label{prop:EDBA2}
Consider the differential equation \eqref{EDBA2}, with $q_{30}p_{11}\neq 0$. Its solutions are as shown in Figure \ref{fig:UmbilicA2}. 
\end{Proposition}

\begin{figure}[htb]
\centering
\includegraphics[width=0.4\linewidth]{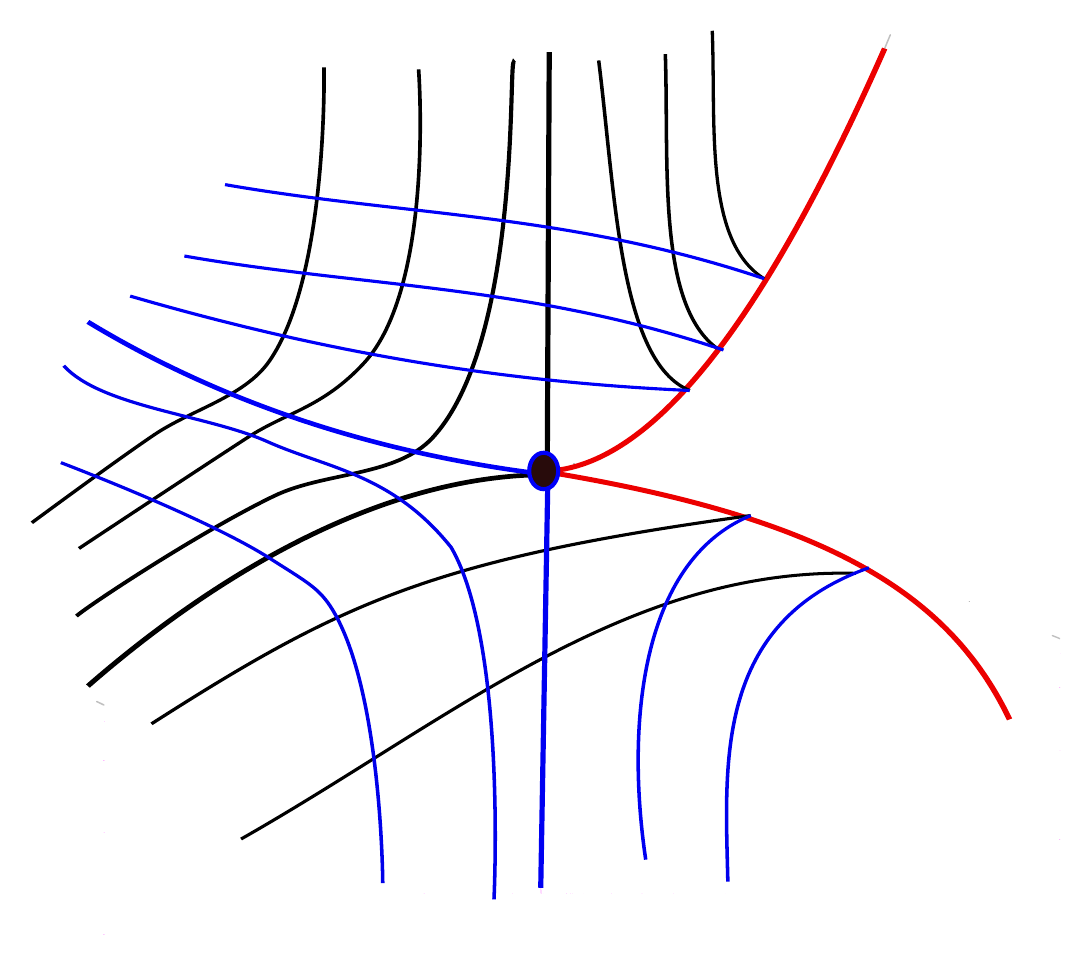}
 \caption{Umbilical point with discriminant having an $A_2$-singularity.}
\label{fig:UmbilicA2}
\end{figure}

\begin{proof}
We shall assume, w.l.o.g., that $p_{11}q_{03}<0$, so that the cuspidal discriminant is contained in the region $x>0$. Consider the blowing-up 
$$
H_p(u,v)=(u^2,u^3v),\ \ \ H_n(u,v)=(-u^2,u^3v), \ \ 0\leq u\leq \epsilon, \ \ -R\leq v\leq R.
$$
Note that each horizontal segment $v=v_0$ in the $(u,v)$-plane correspond to a cusp in the $(x,y)$-plane. 
We shall use $H_p$ to analyze the BDE in the region $x>0$ and $H_n$ to analyze the BDE in the region $x<0$.
The pullback of the BDE by $H_n$ can be written as
$$
\left(p_{11}u^2+O(4)\right)dv^2+\left(12p_{11}uv+O(2)\right)dudv+\left(27p_{11}v^2+2q_{30}+O(1)\right)du^2=0
$$
where here $O(k)$ denote terms of order $\geq k$ in $u$. 
The singular points of this BDE are 
$(u_1,v_1)=\left(0, \sqrt{-\tfrac{2q_{30}}{27p_{11}}}  \right)$ and $(u_2,v_2)=\left(0, -\sqrt{-\tfrac{2q_{30}}{27p_{11}}}  \right)$.
Considering only the first order terms in $u$, the BDE can be factored as
$$
\left(12p_{11}uvdv+(27p_{11}v^2+2q_{30}+\alpha u)du\right)du=0,
$$
where $\alpha$ is some real number that will not be relevant in our analysis. 
The integral lines of the first factor correspond to the integral lines of the vector field
$X=\left( 12p_{11}uv, -27p_{11}v^2-2q_{30}-\alpha u \right)$.
The determinant of $DX(0,v_i)=-648p_{11}^2v^2<0$, and thus both points are saddles. 
A similar calculation shows that the pullback of the BDE by $H_p$ has no singular points. 
To analyze the behavior of the principal lines close to the $x$-axis,  consider the blowing-up $H_1(u,v)=(uv,v)$. Using a similar analysis as above, it follows that $(0,0)$ is a hyperbolic saddle. 
\end{proof}

\section{ Some conjectures and examples}

\subsection{Conjectures for $m\geq 5$ and $m\not\in\mathbb{N}$}

As we have seen above, for a non-degenerate umbilical point with $m\in\{1,2,3,4\}$ the umbilical point is of type $A_m$, for an open and dense subset of $\mathcal{W}$. Following this pattern, we propose the following conjecture:

\smallskip\noindent
{\bf Conjecture:} Consider a general $m\in\mathbb{N}$. Then, for an open and dense subset of $\mathcal{W}$, the non-degenerate umbilical point is of type $A_m$.

In next subsection, we shall show examples of $A_m$-umbilical points, for any $m\in\mathbb{N}$, but we could not prove that they are stable.  In fact, such a proof depends basically on Equation \eqref{Conjecture1}.

On the other hand,
if $m\not\in\mathbb{N}$, Equation \eqref{Conjecture1} is vacuously true, but we could not deal with the calculations of the Taylor expansion of $\delta$, except in the particular case of the examples of next subsection. Nevertheless, we formulate the following conjecture:

\smallskip\noindent
{\bf Conjecture:} Consider $m\not\in\mathbb{N}$. Then, for an open and dense subset of $\mathcal{W}$, the umbilical point is of type $A_{\infty}$.

\subsection{A class of examples with simple calculations}

In this paragraph we consider the particular case $p_{n0}=0$, for any $n\in\mathbb{N}$. As we shall see, in this class Equation \eqref{Conjecture1} holds, the umbilical points are of type $A_m$ if $m\in\mathbb{N}$ and of type $A_{\infty}$ if $m\not\in\mathbb{N}$, thus supporting the conjectures of the previous subsection. 

\begin{Proposition}\label{prop:Example}
Assume that $p_{n0}=0$, for any $n\in\mathbb{N}$ and consider that $q_{02}+mp_{11}=0$. Then
\begin{enumerate}
\item
$W_{m0}=0$.

\item
$q_{k0}=0,\ \ \ 2\leq k\leq m, $
and, for an open and dense subset of $\{p_{n0}=0,n\in\mathbb{N}\}$, 
$q_{m+1,0}\neq 0. $

\item
$q_{k-1,1}=0,\ \ \ 2\leq k\leq m. $
\end{enumerate}
\end{Proposition}

The proof is by induction. 
Assume that, for $k\leq n$, 
$$
q_{k0}=q_{k-1,1}=0, \ \ q_{02}+(k-1)p_{11}\neq 0.
$$
Observe that these relations hold for $k=2$. 
We can write
\begin{equation}\label{eq:1}
\xi^1=p_{11}xy+y\eta^1(x,y)
\end{equation}
with $\eta^1=\eta^1_x=\eta^1_y=0$ at the origin
\begin{equation}\label{eq:2}
\xi^2=\frac{q_{02}}{2}y^2+y^2\eta^2(x,y)+\frac{q_{n+1,0}}{(n+1)!}x^{n+1}+\frac{q_{n1}}{n!}x^ny+O(n+2),
\end{equation}
with $\eta^2=0$ at the origin.

\begin{lemma}\label{lemma:Example-1}
We have that  
\begin{equation}\label{eq:Wn0}
W_{n0}(0,0)=2p_{11}(q_{02}+np_{11})q_{n+1,0}.
\end{equation}
\end{lemma}

\begin{proof}
From Equation \eqref{eq:1}, the terms involving $\xi^1_x$ and $\xi^1_{xx}$ does not appear in $W_{n0}$. From Equation \eqref{eq:2}, the terms involving
$\xi^2_{xx}$ appear in $W_{n0}$ only with its derivative of order $(n-1)$, which is exactly $q_{n+1,0}$, while the terms involving $\xi^2_{x}$ appear only with its derivative of order $n$, which is again $q_{n+1,0}$.

Thus, for the task of calculating $W_{n0}$, we can consider only the terms $\bar{W}$ of $W$ as follows:
\begin{equation*}
\bar{W}= \xi^2_y\xi^2_{xx}\xi^1_{yy}+2\xi^1_y\xi^2_{xx}\xi^1_{xy}+2\xi^2_x\xi^1_{xy}\xi^2_{yy}-2\xi^2_x\xi^1_{yy}\xi^2_{xy}.
\end{equation*}
Now since $\xi^2_y(0,0)=\xi^2_{xy}(0,0)=0$, the first parcel can be discarded. Moreover, since  $\xi^1_{yy}(0,0)=\xi^2_{xy}(0,0)=0$, the last parcel can also be discarded. The derivative of order $n$ of the third parcel gives $2q_{n+1,0}p_{11}q_{02}$, while the derivative of order $n$ of the second parcel at the origin gives exactly $2np_{11}^2q_{n+1,0}$.  Summing up we get formula \eqref{eq:Wn0}.
\end{proof}

From this lemma, we conclude that $W_{m0}=0$ and $q_{k0}=0$, for $k\leq m$. Moreover, for an open and dense subset of $\{p_{n0}=0,n\in\mathbb{N}\}$, $q_{m+1,0}\neq 0$, thus proving the first and second items of Proposition \ref{prop:Example}. 

\begin{lemma}\label{lemma:Example-2}
If $q_{n+1,0}=0$, then
\begin{equation}\label{eq:Wn1}
W_{n-1,1}=2p_{11}(q_{02}+(n-1)p_{11})q_{n1}.
\end{equation}
\end{lemma}

\begin{proof}
For calculating the derivative of order $(n-1,1)$ of the last parcel of $W$, all derivatives must be applied to $\xi^2_x$ and so it results in $2q_{n1}p_{11}q_{02}$. For the terms involving $\xi^2_{xx}$, by the hypothesis $q_{n+1,0}=0$, we must necessarily apply the $(n-2,1)$-derivative in it. Thus we get $2(n-1)p_{11}^2q_{n1}$.

The derivative of order $(n-1,1)$ of the term involving $\xi^1_x\xi^1_{xx}$ is clearly zero, and so we need only to look at the terms $\bar{W}$ below:
\begin{equation*}
\bar{W}= -\xi^2_y\xi^1_{xx}\xi^2_{yy}-2\xi^1_y\xi^1_{xx}\xi^2_{xy}.
\end{equation*}
For the second parcel, we must apply the $y$-derivative to $\xi^1_{xx}$, and then the $(n-1)$ $x$-derivatives to $\xi^2_{xy}$. But since $\xi^1_y=0$ at the origin, we conclude that the $(n-1,1)$ derivative of this term is zero. For the first parcel, we have to apply the $y$-derivative to $\xi^1_{xx}$ and then any number of $x$-derivatives to $\xi^2_{y}$ result in zero. 
Summing all this parcels, we get formula \eqref{eq:Wn1}.
\end{proof}

From this second lemma, we conclude that $q_{k1}=0$, for $k\leq m$, thus proving the third item of Proposition \ref{prop:Example}. We shall verify below that, for the examples in this paragraph, the discriminant function is of type $A_m$ at the origin.

\begin{Proposition}\label{prop:ExampleDiscriminant}
If $m\in\mathbb{N}$, the origin is a singular point of $\delta$ of type $A_m$. If $m\not\in\mathbb{N}$, the origin is a singular point of $\delta$ of type $A_{\infty}$.
\end{Proposition}

Assume that the $n$-jet of $\delta$ is given by 
$$
\delta_n(x,y)=(p_{11}-q_{02})^2y^2+y^2h(x,y),
$$
where $h(x,y)=\sum_{k=1}^{n-2}h_k(x,y)$ and $h_k(x,y)$ is a homogeneous polynomial of degree $k$. 

\begin{lemma}
We have that
$$
\delta_{n+1,0}(0,0)=4(n+1)p_{11}q_{n+1,0}.
$$
\end{lemma}

\begin{proof}
Recall that
$$
\delta=(\xi^1_x-\xi^2_y)^2+4\xi^1_y\xi^2_x.
$$
For the first parcel, the first derivative is , up to a factor $2$,
\begin{equation}\label{eq:deltax1}
(\xi^1_x-\xi^2_y)(\xi^1_{xx}-\xi^2_{xy})=\xi^1_x(\xi^1_{xx}-\xi^2_{xy})-\xi^1_{xx}\xi^2_y+\xi^2_{xy}\xi^2_y.
\end{equation}
Differentiating the parcels involving $\xi^1_x$ or $\xi^1_{xx}$ at the origin will give always zero.  For the last parcel to become non-zero, we need at least $2n-1$ $x$-derivatives. Since $2n-1>n$ for $n>1$, the first parcel can be discarded.

For the second parcel, we use one derivative for the factor $\xi^1_y$ and $n$ derivatives for $\xi^2_x$. But $\xi^2_{n+1,0}=q_{n+1,0}$, and the lemma is proved.
\end{proof}

\begin{lemma}
We have that
$$
\delta_{n1}(0,0)=2 (2n-m-1)p_{11}q_{n1}.
$$
\end{lemma}

\begin{proof}
For the first parcel, consider the $x$-derivative of delta given by formula \eqref{eq:deltax1}. The parcels $\xi^1_x\xi^1_{xx}$ and $-\xi^1_{xx}\xi^2_y$ can clearly be discarded. For the last parcel in formula \eqref{eq:deltax1} to become non-zero, we need a $y$-derivative for $\xi^2_y$ and it rests $(n-1)$ $x$-derivatives for $\xi^2_{xy}$, obtaining $q_{02}q_{n1}$.  Similarly, for the parcel $-\xi^1_x\xi^2_{xy}$, we need one $y$-derivative for $\xi^1_x$ and $n$ $x$-derivatives for $\xi^2_{xy}$, and so we get $-p_{11}q_{n1}$. Summing up, the first parcel of $\delta$ contributes with
$$
2q_{02}q_{n1}-2p_{11}q_{n1}.
$$
For the second parcel of $\delta$, we need one $x$-derivative for $\xi^1_y$ to become non-zero, and so it rests a $(n-1,1)$-derivative for $\xi^2_x$, in which case we obtain $4np_{11}q_{n1}$. Summing up, we get
$$
2q_{02}q_{n1}+(4n-2)p_{11}q_{n1},
$$
thus proving the lemma.
\end{proof}

From the above lemmas, we obtain that $\delta_{k0}=0$, for $k\leq m$ and generically $\delta_{m+1,0}\neq 0$.  Moreover $\delta_{k1}=0$, for $k\leq m-1$. We conclude that the origin is a singularity of type $A_m$ for $\delta$, thus proving the first assertion of Proposition \ref{prop:ExampleDiscriminant}. 

If $m\not\in\mathbb{N}$, the point is not $A_k$, for any  $k\in\mathbb{N}$, and so is of type $A_{\infty}$, thus proving the second assertion of Proposition \ref{prop:ExampleDiscriminant}. In this case, $\delta=\delta_x=\delta_y=0$ for any point in the $x$-axis. We conclude that, since it is non-degenerate, the origin is an isolated umbilical point, but the points of the $x$-axis outside the origin are discriminant non-umbilical points.

\end{document}